\documentclass[11pt,twoside]{atmp}

\usepackage{amsmath,amssymb}

\usepackage[all]{xy}
\usepackage{color}

\copyrightnotice{2022}{1}{1}{34} 

 \newtheorem{thm}{Theorem}[section]
 
 \newtheorem{lem}[thm]{Lemma}
 \newtheorem{prop}[thm]{Proposition}
 \theoremstyle{definition}
 
 \theoremstyle{remark}
 \newtheorem*{remark}{Remark}
 
 \numberwithin{equation}{section}

 \newcommand{\RM}{\mathbb{R}}
 \newcommand{\bv}{\mathbf{v}}
 \newcommand{\bx}{\mathbf{x}}
 \newcommand{\by}{\mathbf{y}} 
 
 \newcommand{\bn}{\mathbf{n}}
 \newcommand{\NM}{\mathbb{N}}
 \newcommand{\ZM}{\mathbb{Z}}
 \newcommand{\CM}{\mathbb{C}}
  
 \newcommand{\dr}{\operatorname{dr}}
 \newcommand{\curl}{\operatorname{curl}}
 \newcommand{\divv}{\operatorname{div}}
 \newcommand{\real}{\operatorname{Re}}
 \newcommand{\imag}{\operatorname{Im}}
 \newcommand{\vraisup}{\operatorname{vraisup}}
 \newcommand{\ds}{\operatorname{ds}}
 
 \newcommand{\dlambda}{\operatorname{d\lambda}} 
 
 \newcommand{\sign}{\operatorname{sign}}
 \newcommand{\dx}{\operatorname{d\mathbf{x}}}
 \newcommand{\dy}{\operatorname{d\mathbf{y}}}

\begin{document}

\title[No-slip boundary condition]
{No-slip boundary condition for vorticity equation in 2D exterior domains}

\arxurl{arXiv:2209.11188}

\author[Aleksei Gorshkov]{Aleksei Gorshkov}

\address{Lomonosov Moscow State University, GSP-1, Leninskie Gory, Moscow, 119991, Russian Federation}  
\addressemail{alexey.gorshkov.msu@gmail.com}

\begin{abstract}
In this article we derive the no-slip boundary condition for a non-stationary vorticity equation. This condition generates the affine invariant manifold and no-slip integral relations on vorticity can be transferred to a Robin-type boundary condition. 
\end{abstract}

\maketitle

\newpage

\section*{Preliminary}Fluid dynamics equations written in terms of vorticity favourably differ from ones covering velocity evolution. So, for both 2D and 3D Navier-Stokes  systems the vorticity dynamics involves a fewer number of equations rather than in the velocity-pressure formulation. For barotropic fluid the curl removes the most problematic term - pressure. If the system doesn't involve boundary condition (e.g. Cauchy problem), then the $\curl$ operator significantly simplifies corresponding {{\it initial-value problem}. But in the case of {\it initial-boundary-value problem} it isn't true since Dirichlet no-slip boundary turns out to be a sophisticated integral condition. And so, vorticity boundary conditions mostly are deduced in vorticity-stream formulation\cite{Wu}\cite{And}\cite{WJ}\cite{Suh}. In \cite{Zakh} was studied the parametric boundary condition covering vorticity and stream. Some new results on Newman and Dirichlet boundary conditions for vorticity on solid walls were given in \cite{OH}.

We will investigate the boundary vorticity which corresponds to the no-slip condition. At first sight there is no analogous to no-slip condition only in terms of vorticity without additional functions involved (e.g. stream function, velocity). Biot-Savar law restores the solenoidal velocity field $\bv(\bx)$ induced by vorticity $w(\bx)$. It expresses vorticity via velocity field by some integral relation. If the flow interacts with solid by no-slip condition, then it turns to zero integral relation in Biot-Savar law which doesn't admit explicit integration. If we look at dolphins, then their skin can read vorticity distribution in order to prevent turbulence. In a playful way we can say that dolphins know something about boundary vorticity. 

In this paper will be established vorticity boundary condition for both linear and nonlinear Helmholtz equations without any stream function involved in. For the linear Stokes system in the exterior of the disc such boundary condition along with explicit formula to Stokes problem was obtained by the author in \cite{AG}. In this article we show that for nonlinear vorticity equation in the exterior of the disc of radius $r_0$ the boundary condition can be defined in terms of vorticity Fourier coefficients $w_k(t,r)$ as Robin-type boundary problem: 
\begin{align}\label{_3:bound:nonlin}
\frac{\partial w_k(t,r)}{\partial r}\Big|_{r=r_0} +  |k| w_k(t,r_0) = u_k(t).
\end{align}
And the same boundary-value problem will be actual in more general domains with help of Riemann mapping.

For exterior flows it is typical to study solutions with infinite energy. These solutions with finite Dirichlet integral were studied in \cite{Abe}\cite{MS}. For Cauchy problem global existence and bounds on velocity as time $t\to \infty$ were established in \cite{GW1}\cite{Zel}\cite{Th}. In this paper we will research the vorticity equation for infinite energy solutions. The main difficulty is the fact that the space $L_2$ is not enough to describe vorticity dynamics. The solutions with finite vorticity energy but nonzero total circulation may have infinite kinetic energy. But in case of zero circularity it's not getting  considerably better. In order to restore velocity via Biot-Savar law and correctly define the nonlinear term $(\bv, \nabla w)$ we need to impose additional requirements on the phase space. 

The Laplace operator with Robin boundary condition (\ref{_3:bound:nonlin}) possesses the non-trivial kernel and as result it causes the presence of the stationary solution for the corresponding evolution equation. But from the no-slip condition follows an orthogonality relation between the solution and the kernel (see \cite{AG} for more details). It is entirely consistent with the Stokes Paradox.

The paper is arranged as follows. In Section 1 we study Biot-Savar law and the integral relation for no-slip condition. Then in Section 2 we will derive the precise integral boundary condition for Stokes system and its local approximation for Helmholtz vorticity equation. In Section 3 we will prove the local existence of no-slip condition for vorticity. In Appendix the solvability of the vorticity equation with boundary (\ref{_3:bound:nonlin}) will be established.

{\bf Notations:}
In the paper we will exploit $\RM^2$ as real as well as a complex plane depending on the context.
For points from $\RM^2$ we will use different real and complex notations including polar coordinates $r,\varphi$ such as $\bx=(x_1,x_2)$, $z=x_1+ix_2=r e^{i\varphi}$. For functions defined on $E\subset \RM_+$ along with classic spaces $L_p(E)$, $L_p(E)$ we will use $L_p(E; r)$, $L_p(E; \lambda)$ of square-integrable functions with  infinitesimal elements $r\dr$, $\lambda \dlambda$, supplied with norms
\begin{align*}
\|f\|^p_{L_p(E; r)}=\int\limits_E |f(r)|^p r\dr, \\
\|f\|^p_{L_p(E; \lambda)}=\int\limits_E |f(\lambda)|^p \lambda \dlambda.
\end{align*}

Velocity $\bv$ will be used in both Cartesian $(v_1, v_2)$ and polar coordinates $(v_r, v_\phi)$. Fourier coefficients for function $f$ will be referred $f_k$ with prefix $k$. For Laplace operator its Fourier expansion will involve $\Delta_k$ defined as
\begin{align} \label{fourierlaplace}
\Delta_k w(t,r) = \frac 1r \frac {\partial}{\partial r}\left(r \frac {\partial}{\partial r}w(t,r)\right) - \frac{k^2}{r^2} w(t,r).
\end{align}
 
\section{Biot-Savar law in exterior domains and no-slip integral condition}

Now we study when the solenoidal velocity field $\bv(\bx)$ can be uniquely restored from its vorticity $w(\bx)$. Consider the following elliptic problem in exterior domain $\Omega$:
\begin{eqnarray}
&&\rm{div}~ \bv(\bx) = 0, \label{freediv} \\
&&\rm{curl}~ \bv(\bx) = w(\bx), \label{curleq} \\
&&\bv(\bx)=0,~\bx \in \partial \Omega, \label{bound}\\
&&\bv(\bx)\to\bv_\infty,~|\bx|\to \infty, \label{boundinf}.
\end{eqnarray}

Exterior domains are not simply connected and the problem above could haven't unique solution. For example, equations (\ref{freediv}), (\ref{curleq}) supplied with slip condition on the boundary 
\begin{equation}\label{slip}
\left(\bv(\bx), \bn \right) = 0,~\bx \in \partial \Omega,
\end{equation}
and fixed flow at infinity (\ref{boundinf}) have a unique solution only if we fix circularity at infinity ($\bn$ is an outer normal to boundary). No-slip condition (\ref{bound}) is stronger than (\ref{slip}), and so some additional restrictions on $w(\bx)$ are required. These restrictions can be realized via moment relations for vorticity. In \cite{AG} the solvability of the system above was researched in detail for slip and no-slip conditions in the exterior of the disc. Here we extend these results on more general domains and obtain integral no-slip condition.

So, we need to fix circularity at infinity. From a physical point of view it is natural to suppose zero-circularity:
\begin{equation} \label{zerocirculation}
\lim_{R\to\infty}\oint_{|\bx|=R} \bv \cdot d\mathbf{l} = 0.
\end{equation} 

Then the solution of the above problem if it exists is given by Biot-Savar formula
\begin{equation} \label{BSformula}
\bv(\bx) =\frac 1{2\pi} \int_\Omega \frac{(\bx-\by)^\perp}{|\bx-\by|^2} w(\by) \operatorname{d\by} + \bv_\infty,
\end{equation}
which we rewrite in polar coordinates. Boundary condition (\ref{bound}) obliges the vorticity $w(\by)$ to satisfy some additional equities. 

The relationship between Cartesian and polar coordinate systems for $\bv_\infty=(v_{\infty,x},v_{\infty,y})$ is given by formulas:
\begin{align*}
&v_{\infty,r}=v_{\infty,x}\cos \varphi + v_{\infty,y}\sin \varphi, \\ 
&v_{\infty,\phi}=v_{\infty,y}\cos \varphi - v_{\infty,x}\sin \varphi. 
\end{align*}

Then its Fourier coefficients are determined as
\begin{align}
&v_{\infty,r,k}=\frac{\delta_{|k|,1}}2 (v_{\infty,x} - i k v_{\infty,y}) \label{fouriercoeffr}, \\ 
&v_{\infty,\phi,k}=\frac{\delta_{|k|,1}}2 (v_{\infty,y} + i k v_{\infty,x}) \label{fouriercoeffphi},
\end{align} 
and
\begin{align} \label{vrvphi}
v_{\infty,\phi,k} =  \sign(k) iv_{\infty,r,k}.	
\end{align}

All Fourier coefficients of the external flow equal to zero except $k=\pm 1$. For horizontal flow $\bv_\infty=(v_\infty,0)$
\begin{align*}
&v_{\infty,r,k}=\frac{\delta_{|k|,1}}2 v_\infty, \\ 
&v_{\infty,\phi,k}=i k\frac{\delta_{|k|,1}}2  v_\infty.
\end{align*}

\subsection{Biot-Savar law in exterior of the disc}
In this subsection the domain under investigation will be the exterior of the disc $B_{r_0}=\{\bx \in \RM^2,~|\bx| > r_0 \},~r_0>0$. We will derive Biot-Savar law and no-slip boundary condition for Stokes and Navier-Stokes systems in integral form. 

In polar coordinates equations (\ref{freediv}),(\ref{curleq}) can be written in Fourier coefficients $v_{r,k}$, $v_{\varphi,k}$:
\begin{eqnarray*}
&&{\frac {1}{r}}{\frac {\partial }{\partial r}}\left(rv_{r,k}\right)+{\frac {ik}{r}} v_{\varphi,k} = 0,\\
&&{\frac {1}{r}}{\frac {\partial }{\partial r}}\left(rv_{\varphi,k}\right)-{\frac {ik}{r}} v_{r,k} = w_k.
\end{eqnarray*}

The basis for solutions of homogeneous system when $w_k \equiv 0$ consists of two vectors:
\begin{align*}
\begin{pmatrix}
v^1_{r,k} \\
v^1_{\varphi,k} 
\end{pmatrix} 
=
\begin{pmatrix}
ir^{-k-1} \\
r^{-k-1} 
\end{pmatrix}
, \\
\begin{pmatrix}
v^2_{r,k} \\
v^2_{\varphi,k} 
\end{pmatrix} 
=
\begin{pmatrix}
ir^{k-1} \\
-r^{k-1} 
\end{pmatrix}.
\end{align*}

The solution of this system with boundary relations (\ref{bound}), (\ref{boundinf}) and zero-circularity condition (\ref{zerocirculation}) was derived in \cite{AG} as Biot-Savar law in exterior of the disc in the following form for $k \in \ZM$:
\begin{eqnarray} \label{BiotSavar1}
&&v_{r,k} = \sign(k)  \frac{ir^{-|k|-1}}2 \int_{r_0}^r s^{|k|+1}w_k(s)\ds \\ 
&&~~~~~~~~~~~~~~~~~~+ \sign(k) \frac{ir^{|k|-1}}2 \int_r^\infty s^{-|k|+1}w_k(s)\ds  + v_{\infty,r,k} \nonumber \\ \label{BiotSavar2}
&&v_{\varphi,k} = \frac{r^{-|k|-1}}2 \int_{r_0}^r s^{|k|+1}w_k(s)\ds \nonumber \\ 
&&~~~~~~~~~~~~~~~~~~- \frac{r^{|k|-1}}2 \int_r^\infty s^{-|k|+1}w_k(s)\ds + v_{\infty,\phi,k}.
\end{eqnarray} 

No-slip condition (\ref{bound}) and (\ref{vrvphi}) lead to moment relations for vorticity ($k\in \ZM$): 
\begin{equation} \label{noslipcondintegral}
\int_{r_0}^\infty s^{-|k|+1}w_k(s)\ds = 2ik v_{\infty,r,k} = 2 v_{\infty,\phi,k}. 
\end{equation}
From (\ref{fouriercoeffr}), (\ref{fouriercoeffphi}) these moments don't equal to zero only if $|k|=1$. 


The above formulas (\ref{BiotSavar1}), (\ref{BiotSavar2}) represent Fourier coefficients of Biot-Savar formula (\ref{BSformula}).


Since for $p >1$  $\nabla \bv$ is obtained from $\omega$ via a singular integral kernel of Calderon-Zygmund type\cite{CZ}\cite{St}, then  
\begin{align}\label{bsest2}
\|\nabla \bv(\cdot)\|_{L_p} \leq C \| w \|_{L_p}. 
\end{align} 

The case $p=2$ causes most difficulties in estimates of Biot-Savar law. The following lemma gives an estimate for $w\in H^1$.
\begin{lem}\label{bsest}
Let $w(\cdot) \in L_2(B_{r_0})$, the Fourier coefficients at $k=-1, 0, 1$ belong to $L_1(r_0,\infty)$, $\bv(\cdot)$ - be the solution of (\ref{freediv}) - (\ref{boundinf}), (\ref{zerocirculation}). Then the following estimate holds
with some $C>0$:
\begin{align*}
&\vraisup_{r\in [r_0,\infty)} \| \bv(r,\cdot) - \bv_\infty \|^2_{H^{1/2}(S_{r})}  \leq  \\ &~~~~~~~~~~~C
\left (\|w(\cdot)\|^2_{L_2(B_{r_0})} + \sum_{k=-1,0,1}\|w_k(\cdot)\|^2_{L_1(r_0,\infty)} \right),
\end{align*}
where $S_r = \{\bx \in \RM^2,~|\bx|=r\}$.
\end{lem}

\begin{proof}
We will estimate one of the terms in (\ref{BiotSavar1}) when others can be processed in a similar way. 

For $|k| > 1:$ 
\begin{align}\label{bsest_proof}
\left | r^{|k|-1} \int_r^\infty s^{-|k|+1}w_k(s)\ds \right |^2 \leq \frac {\|w_k\|^2_{L_2(r_0,\infty; r)}}{2|k|-2}.
\end{align}

For $|k| = 1:$
\begin{align*}
\left | \int_r^\infty w_{\pm 1}(s)\ds \right |^2 \leq (1+1/r_0) \|w_{\pm 1}\|^2_{L_1(r_0,\infty, r)}.
\end{align*}

For $|k| =0:$ 
\begin{align*} 
\left | r^{-1} \int_r^\infty s w_0(s)\ds \right |^2 \leq  r_0 \|w_0\|^2_{L_1(r_0,\infty; r)}.
\end{align*}


Fractional differentiation of order $\frac 12$ corresponds to multiplier $\sqrt k$ for Fourier coefficients. Summarizing by $k$ we obtain the required estimate. 
%
%
\end{proof}

We rewrite moment relationship in terms of vorticity $w(\bx)$ for $k\geq 0$:
\begin{align*}
\int_{r_0}^\infty s^{-|k|+1}w_k(s)\ds = \frac 1{2 \pi} \int_{r_0}^\infty \int_0^{2\pi} s^{-k+1}w(s,\varphi)e^{-ik\varphi}\ds d\varphi \\ =\frac 1{2 \pi} \int_{B_{r_0}} \frac {w(\bx)}{z^k} \dx = 2ik v_{\infty,r,k},
\end{align*} 
where $z=x_1+ix_2=s e^{i\varphi}$.	

For $k<0$ we have absolutely the same moments equity:
\begin{align*} 
\int_{r_0}^\infty s^{-|k|+1}w_k(s)\ds = \frac 1{2 \pi} \int_{r_0}^\infty \int_0^{2\pi} s^{k+1}w(s,\varphi)e^{-ik\varphi}\ds d\varphi \nonumber \\ =\frac 1{2 \pi} \int_{B_{r_0}} \overline{z^k} w(\bx) \dx = 2ik v_{\infty,r,k},~k<0.
\end{align*}
The last formula for $k<0$ is just complex conjugation of the analogous formula for $k \geq 0$ due to equities
$$
v_{\infty,r,k} = \overline{v_{\infty,r,-k}},  v_{\infty,\phi,k} = \overline{v_{\infty,\phi,-k}}. 
$$

Then  for no-slip condition we have affine subspace $M$ which must be invariant under vorticity flow:
\begin{align}\label{noslipcondintegral2} 
M=\left \{ w(\bx)\in L_1(B_{r_0})~\Big |~ 
\int_{B_{r_0}} \frac {w(\bx)}{(x_1+ix_2)^k} \dx =  
4 \pi ik v_{\infty,r,k},~k \in \ZM_+  \right \}.
\end{align} 
Note that from (\ref{fouriercoeffr}), (\ref{fouriercoeffphi}) the integral relations in the definition of $M$  are non-zero ones only if $k=1$. 

For the Fourier coefficients the invariance of $M$ means that for vorticity flow which is described by the coefficients $w_k(t,\cdot)$ holds
$$
\int_{r_0}^\infty s^{-|k|+1}w_k(t,s)\ds = const,~k \in \ZM.
$$

In \cite{AG} was proved
\begin{thm}[Biot-Savart Law in polar coordinates]\label{Biotpolarnoslip}
If $w(\bx) \in M$ then there exists the unique solution of (\ref{freediv}) - (\ref{boundinf}), (\ref{zerocirculation}) given by (\ref{BiotSavar1}), (\ref{BiotSavar2}) 
with Fourier coefficients $v_{r,k}$, $v_{\varphi,k} \in L_\infty(r_0,\infty)$.
\end{thm}

\subsection{Biot-Savar Law in the simply connected domains}
Here we derive Biot-Savar law in more general domains. We will establish that the Robin-type boundary (\ref{robin_bound}) stays actual for these domains.
Let $\Omega=\RM^2 \setminus B$, where $B$ is bounded simple-connected domain with smooth boundary and $\Phi$ be a Riemann mapping from $\Omega$ into exterior of the disc $B_{r_0}$ such that
\begin{align*}
\Phi(z)=z+O\left (\frac 1z \right ), \\
\Phi'(z)=1+O\left (\frac 1{z^2} \right ).
\end{align*}

Then $\bv=\bv(\Phi^{-1}(z))=\bv(\real \Phi^{-1}(x_1+ix_2), \imag \Phi^{-1}(x_1+ix_2))$ defines vector field in $B_{r_0}$. We will not use the tensor form of divergence and consider the vector field $(v_r(\bx),v_\phi(\bx))$ as a set of two scalar functions. Then with help of Cauchy–Riemann relationship the equations (\ref{freediv}), (\ref{curleq}) turn to relation
\begin{equation}
\curl \bv + i \divv \bv = \frac {\Phi'(z)}{|\Phi'(z)|^2}w.
\end{equation}

Then the system (\ref{freediv})-(\ref{boundinf}) after Riemann mapping in $B_{r_0}$ takes the form
\begin{eqnarray}
&&\rm{div}~ \bv(\bx) = \imag \overline {\Phi'^-1(z)} w(\bx)  \label{freediv3}\\
&&\rm{curl}~ \bv(\bx) = \real \overline {\Phi'^{-1}(z)} w(\bx)  \label{curleq3}\\
&&\bv(\bx)=0,~|\bx|=r_0  \label{bound3}\\
&&\bv(\bx)\to\bv_\infty,~|\bx|\to \infty.  \label{boundinf3}
\end{eqnarray}

Let
\begin{align*}
r_k(r)=[\imag \overline {\Phi'^{-1}(z)} w(\bx)]_k,\\
q_k(r)=[\real \overline {\Phi'^{-1}(z)} w(\bx)]_k,
\end{align*}
where subscript $k$ denotes $k$-th Fourier harmonic and $z=re^{i\varphi}=x_1+ix_2$.

Rewrite (\ref{freediv3}),(\ref{curleq3}) in polar coordinates in terms of Fourier coefficients $v_{r,k}$, $v_{\varphi,k}$:
\begin{eqnarray*}
&&{\frac {1}{r}}{\frac {\partial }{\partial r}}\left(rv_{r,k}\right)+{\frac {ik}{r}} v_{\varphi,k} = r_k(r),\\
&&{\frac {1}{r}}{\frac {\partial }{\partial r}}\left(rv_{\varphi,k}\right)-{\frac {ik}{r}} v_{r,k} = q_k(r).
\end{eqnarray*}

Under the assumption of zero-circularity (\ref{zerocirculation}) from Stoke's theorem we have 
\begin{align*}
\lim_{R\to\infty}\oint_{|\bx|=R} \bv(\bx) \cdot d\mathbf{l}= \frac 1{2 \pi} \int_{\Omega} w(\bx) \dx \\ = \frac 1{2 \pi} \int_{B_{r_0}} \frac {w(\bx)}{|\Phi'|^2} \dx = 
\int_{r_0}^\infty  \left [\frac {w(\bx)}{|\Phi'|^2} \right ]_{k=0} s \ds = 0.
\end{align*}

It guarantees the uniqueness of the above system. Existence will be provided by the moment relations below. The solution of the above system for $k \in \ZM$ is derived by the similar formulas (\ref{BiotSavar1}), (\ref{BiotSavar2}):
\begin{eqnarray} \label{BiotSavar1_2}
&&v_{r,k} = \sign(k)  \frac{ir^{-|k|-1}}2 \int_{r_0}^r s^{|k|+1}(q_k-i\sign(k) r_k)\ds \\ 
&&~~~~~~~~~~~~~~~~~~+ \sign(k) \frac{ir^{|k|-1}}2 \int_r^\infty s^{-|k|+1}(q_k+i\sign(k)r_k)\ds  + v_{\infty,r,k} \nonumber \\ \label{BiotSavar2_2}
&&v_{\varphi,k} = \frac{r^{-|k|-1}}2 \int_{r_0}^r s^{|k|+1}(q_k-i\sign(k)r_k)\ds \nonumber \\ 
&&~~~~~~~~~~~~~~~~~~- \frac{r^{|k|-1}}2 \int_r^\infty s^{-|k|+1}(q_k+i\sign(k)r_k)\ds + v_{\infty,\phi,k}.
\end{eqnarray}

Formulas (\ref{BiotSavar1_2}), (\ref{BiotSavar2_2}) combined with (\ref{bound}) lead to the relations on vorticity ($k\in \ZM$): 
\begin{equation} \label{noslipcondintegralrieman}
\int_{r_0}^\infty s^{-|k|+1}\left ( q_k(s)+i\sign(k)r_k(s) \right )\ds = 2ik v_{\infty,r,k} = 2 v_{\infty,\phi,k}. 
\end{equation}

The above formulas (\ref{BiotSavar1_2}), (\ref{BiotSavar2_2}) are the Fourier coefficients of Biot-Savar formula 
\begin{equation}\label{_1:BS}
\bv(\bx)=\frac 1{2\pi} \int_{B_{r_0}} \left ( \frac{(\bx-\by)^\perp}{|\bx-\by|^2} \real \overline {\Phi'^{-1}(z)} +
\frac{\bx-\by}{|\bx-\by|^2} \imag \overline {\Phi'^{-1}(z)} \right ) w(\by)d\by +\bv_\infty.
\end{equation}

Following by the same way as in Theorem \ref{invthm} we define affine subspace $M$ via  vorticity moments which must be invariant under vorticity flow:
\begin{align}\label{noslipcondintegral2_omega} 
M=\Big \{ w(\bx)\in L_1(B_{r_0})~\Big |~ 
\int_{B_{r_0}}\overline{\Phi'^{-1}(x_1+ix_2)} \frac {w(\bx)}{(x_1+ix_2)^k} \dx = \nonumber \\ 
4 \pi ik v_{\infty,r,k},~k \in \ZM_+  \Big \}.
\end{align} 
In view of (\ref{fouriercoeffr}), (\ref{fouriercoeffphi}) all moments in the definition of $M$ must be equal to zero except $k=1$. 

\begin{prop}[Biot-Savart Law in exterior domain]\label{Biotpolarnoslip2}
If $w(\Phi^{-1}(z)) \in M$ then there exists the unique solution of (\ref{freediv}) - (\ref{boundinf}), (\ref{zerocirculation}) given by (\ref{BiotSavar1_2}), (\ref{BiotSavar2_2}) with Fourier coefficients $v_{r,k}$, $v_{\varphi,k} \in L_\infty(r_0,\infty)$.
\end{prop}
\begin{proof}All we need is to prove the validity of no-slip condition. Set $r=r_0$ in (\ref{BiotSavar1_2}), (\ref{BiotSavar2_2}). Since with $z=se^{i\varphi}=x_1+ix_2$
$$
q_k(s)+ir_k(s) = [\overline{\Phi'^{-1}(z)} w(x)]_k,
$$
equities (\ref{BiotSavar1_2}), (\ref{BiotSavar2_2}) can be written in terms of $w(\bx)$. Then
\begin{align} \label{_1:BiotSavar_int}	
&\int_{r_0}^\infty s^{-|k|+1}\left ( q_k(s)+ir_k(s) \right )\ds \\&= \frac 1{2 \pi} \int_{r_0}^\infty \int_0^{2\pi} s^{-|k|+1}\overline{\Phi'^{-1}(z)}w(s,\varphi)e^{-ik\varphi}\ds d\varphi \nonumber \\ &=\frac 1{2 \pi} \int_{B_{r_0}} \overline{\Phi'^{-1}(z)} \frac {w(\bx)}{z^k} \dx = 2ik v_{\infty,r,k}. \nonumber
\end{align}
\end{proof}

\begin{lem} \label{lembsest} In the exterior simple-connected domain $\Omega$ with smooth boundary let $\Phi$ be a Riemann mapping from $\Omega$ into exterior of the disc $B_{r_0}$, $w(\cdot) \in L_1(B_{r_0})\cap L_2(B_{r_0})$, $\bv=\bv(\Phi^{-1}(z))=\bv(\real \Phi^{-1}(x_1+ix_2), \imag \Phi^{-1}(x_1+ix_2))$ defines vector field in $B_{r_0}$ of the the solution of (\ref{freediv}) - (\ref{boundinf}), (\ref{zerocirculation})  with vorticity $w(\Phi^{-1}(z))$. Then the following estimate holds with some $C>0$:
\begin{align*}
\vraisup_{r\in [r_0,\infty)} \| \bv(r,\cdot) - \bv_\infty \|^2_{H^{1/2}(S_{r_0})}  \leq C
(\|w(\cdot)\|_{L_2(B_{r_0})} + \|w(\cdot)\|_{L_1(B_{r_0})}). 
\end{align*}
\end{lem}

Since $|\Phi'(z)|$ is bounded, then the proof of this lemma is the same as for Lemma \ref{bsest} in the previous subsection.

\subsection{Invariant affine manifolds for a no-slip condition}
Consider the flow of 2D velocity field $\bv(t,\bx)$ and its vorticity $w(t,\bx)=\curl \bv(t,\bx)$. The set of relations (\ref{noslipcondintegral2_omega}) is the integral form of the no-slip condition (\ref{bound}).

\begin{prop} \label{invthm}
Given initial datum $\bv_0(\bx)$ satisfying no-slip condition (\ref{bound}), infinity condition (\ref{boundinf}), zero-circularity (\ref{zerocirculation}), such that \\$w_0=\curl \bv_0(\Phi^{-1}(z))$ $\in L_1(\Omega)$, and $w(t,\cdot)=\curl \bv(\Phi^{-1}(z))$ $\in L_1(\Omega)$ be the vorticity flow. Then $w_0 \in M$ and in order to conserve no-slip condition the affine subspace $M$ must be invariant under the flow, e.g. for any time $t>0$ $w(t,\cdot) \in M$.
\end{prop}

\begin{proof}
Since $\bv_0(\bx)$ satisfies no-slip condition, then from (\ref{BiotSavar1_2}), (\ref{BiotSavar2_2}), (\ref{_1:BiotSavar_int}) follows $w(t,\cdot) \in M$.  
\end{proof}

If $\bv_\infty=0$ then this affine manifolds $M$ path through zero and becomes the invariant subspace. If we limit infinite set of relations in (\ref{noslipcondintegral2_omega}) only by $k=1,...,N$, we obtain affine manifold $M_N$ of finite codimension:  
\begin{align}\label{noslipcondintegral2_omega_N} 
M_N=\Big \{ w(\bx)\in L_1(B_{r_0})~\Big |~ 
\int_{B_{r_0}}\overline{\Phi'^{-1}(z)} \frac {w(\bx)}{z^k} \dx =  
4 \pi ik v_{\infty,r,k},\\ k =0, 1, \dots, N  \Big \}. \nonumber
\end{align} 

The following lemma says that since $M$ corresponds to a no-slip condition, then $M_N$ is the approximation of this boundary condition. 
\begin{lem}Let for any fixed $t\in [0,T]$ $w_N(t,\cdot)$ are uniformly bounded in $H^1(\Omega)$ by $N$. Then the vector field $\bv^N(t,\cdot)$ given by (\ref{_1:BS}) converges weakly in $H^{1/2}(\partial \Omega)$ to zero for $t\in [0,T]$ as $N\to \infty$. 
\end{lem}

\begin{proof}
Without loss of generality assume that $\Omega=B_{r_0}$.
\begin{align*} 
&v_{r,k}(t,r_0) = \sign(k) \frac{ir^{|k|-1}}2 \int_{r_0}^\infty s^{-|k|+1}w_k(t,s)\ds  + v_{\infty,r,k}  \\ 
&v_{\varphi,k}(t,r_0) = \frac{r^{|k|-1}}2 \int_{r_0}^\infty s^{-|k|+1}w_k(t,s)\ds + v_{\infty,\phi,k}.
\end{align*} 

From $w(t,\cdot) \in M_N$ $v_{r,k}(t,r_0) = v_{\varphi,k}(t,r_0)=0$ for $k =-N, \dots, N$. For $k>N$
\begin{align*}
\left | \int_{r_0}^\infty s^{-|k|+1}w_k(t,s)\ds  + v_{\infty,r,k} \right |=
\left | \int_{r_0}^\infty s^{-|k|+1}w_k(t,s)\ds \right |.
\end{align*}

Then from (\ref{bsest_proof})
\begin{align*}
\| \bv^N(t,\cdot)\|_{H^{1/2} (\partial \Omega)}   \leq C \sum_{k=N+1}^\infty \|w^N_k(t,\cdot)\|^2_{L_2(r_0,\infty, rdr)}. 
\end{align*}

For $w_N(t,\cdot) \in M_N$ the Fourier coefficients of $\bv_r$, $\bv_\phi$ with index $k =-N, \dots, N$ equal to zero  which implies 
\begin{align}\label{weakconv}
\bv(t,\bx') \rightharpoonup 0,~\bx' \in \partial \Omega,~(weakly)
\end{align}
in $H^{1/2}(\partial \Omega)$. 
\end{proof}

\begin{remark} If $w_N(t,\cdot)$ are uniformly bounded in $H^2$ one can prove that $v(t,\bx')$ converges weakly in $H^{3/2}(\partial \Omega)$ to zero. And from the Rellich–Kondrachov Theorem due to the compact embedding follows that for any $t\in [0,T]$ 
\begin{align}\label{strongconv}
\bv(t,\bx') \to 0,~\bx' \in \partial \Omega~(strongly)
\end{align}
in $L_2(\partial \Omega)$.

\end{remark}

\section{No-slip integral condition for vorticity in exterior domains}
Consider the initial-boundary-value problem for the Navier-Stokes system defined in exterior domain $\Omega$ modelling flow around solid with given constant horizontal flow at infinity $\bv_\infty = (\bv_{1,\infty},\bv_{2,\infty}) \in \RM^2$:
\begin{eqnarray}
&&\partial_t \bv - \Delta \bv +(\bv,\nabla)\bv = \nabla p \label{maineqns}\\
&&{\rm div}~\bv(t,\bx)=0 \label{freedivns}\\
&&\bv(0,\bx)=\bv_0(\bx) \label{initns}\\
&&\bv(t,\bx)=0,~|\bx|=r_0 \label{boundns}\\
&&\bv(t,\bx) \to \bv_\infty,~|\bx|\to \infty. \label{boundinfns}
\end{eqnarray}
Here $\bv(t,\bx)=(v_1(t,\bx),v_2(t,\bx))$ is the velocity field and $p(t,\bx)$ is the pressure.

Applying the curl operator  
$w(t,\bx)=$ ${\rm curl}~\bv(t,\bx)$ $=\partial_{\bx_1}v_2 -  \partial_{\bx_2}v_1$ we get boundary problem for vorticity equation
\begin{eqnarray} 
\frac{\partial w(t,\bx)}{\partial t}	 - \Delta w + (\bv,\nabla)w  = 0,  \label{maineqw} \\
w(0,\bx)=w_0(\bx) \label{initw}\\
\curl^{-1} w(t,\bx) \Big|_{|\bx|=r_0} = 0, \label{boundw}\\
w(t,\bx) \to 0,~|\bx|\to \infty \label{boundinfw}
\end{eqnarray} 
with initial datum $w_0(\bx)={\rm curl}~\bv_0(\bx)$. 

Vector field $\bv(t,\bx)$ can be derived from $w(t,\bx)$ using Green function $G(\bx,\by)$ for Laplace operator $\Delta$:
$$
\bv(t,\bx) = \int_\Omega \nabla_x^\perp G(\bx,\by) w(\by) \dy +  \bv_\infty,
$$
where 
\begin{equation}\label{_2:GreenProp}
\Delta_x G(\bx,\by)=0,~\bx\neq \by.
\end{equation}

Indeed
$$
{\rm div}~\bv(t,\bx) = \int_\Omega (\nabla_x,\nabla_x^\perp) G(\bx,\by) w(\by) \dy = 0,
$$
and
$$
{\rm curl}~\bv(t,\bx) = \int_\Omega \Delta_x G(\bx,\by) w(\by) \dy  = w.
$$

Then no-slip condition gives the following integral expression
$$
\bv(t,\bx') = \int_\Omega \nabla_x^\perp G(\bx',\by) w(t,\by) \dy = 0,~\bx'\in \partial \Omega,~\forall t>0.
$$

Velocity field for Stokes system satisfies
\begin{align} \label{stokeseqvelocity}
\frac{\partial v(t,\bx)}{\partial t} - \Delta v(t,\bx)  = \nabla p
\end{align}
when vorticity evolution for Stokes system is described by the heat equation  
\begin{align} \label{stokeseq}
\frac{\partial w(t,\bx)}{\partial t} - \Delta w(t,\bx)  = 0.
\end{align}

Multiplying it by $\nabla_x^\perp G(\bx',\by)$ and integrating over exterior domain using Green formula with help of (\ref{_2:GreenProp}) we will have integral boundary condition 
$$
\int_{\partial \Omega} \left (
\nabla_x^\perp G(\bx',\by) \frac{\partial w(t,\by)}{\partial n} -w(t,\by) \frac{\partial }{\partial n}\nabla_x^\perp G(\bx',\by) \right ) \dy=0,~\forall \bx' \in \partial \Omega.
$$
It is still an integral condition but only with the surface integral over $\partial \Omega$ involved. 

For cylindrical domains this surface integral turns into boundary condition on Fourier coefficients. In this section for 2D Stokes and Navier-Stokes we derive boundary condition in terms of Fourier harmonics in exterior simply-connected domains. 

\subsection{Linear vorticity equation in the exterior of the disc} 

Consider Stokes flow for vorticity (\ref{stokeseq}) and supply it with Robin-type boundary condition:
\begin{equation}\label{robin_bound}
r_0\frac{\partial w_k(t,r)}{\partial r}\Big|_{r=r_0} + |k| w_k(t,r_0) = 0,~k \in \ZM
\end{equation}
and
\begin{equation}
w(t,\bx) \to 0,~|\bx|\to \infty. \label{boundinfvorticity}
\end{equation}

Then $M$ is invariant under the flow $w(t,\cdot)$. Indeed, fix $k>0$ and divide equation (\ref{stokeseq}) by $z^k$ and integrate over the exterior of the disc $B_{r_0}$. From moment relations (\ref{noslipcondintegral2}) follows
$$
\frac d{dt}\int_{B_{r_0}} \frac {w(t,\bx)}{z^k}  \dx = 0
$$ 
and thus
$$
\int_{B_{r_0}}  \frac {\Delta w}{z^{k}}\dx=0.
$$

In other hand
\begin{align}\label{integrationbypart_dissip}\nonumber
\int_{B_{r_0}} \frac {\Delta w}{z^{k}}\dx = \int_{r_0}^\infty \int_0^{2\pi} \frac {\Delta w}{s^ke^{ik\phi}}sdsd\varphi = 2\pi \int_{r_0}^\infty s^{-k+1}\Delta_k w_k(t, s) ds \\ \nonumber=-2\pi \int_{r_0}^\infty s^{-|k|} \left (  \frac {\partial}{\partial s}\left(s \frac {\partial}{\partial s}w_k(t,s)\right) - \frac{k^2}{s}  w_k(t,s) \right ) \ds \\ \nonumber=
- r_0^{-|k|+1}\frac{\partial w_k(t,r)}{\partial r}\Big|_{r=r_0} + \int_{r_0}^\infty s^{-|k|} \left (  |k|  \frac {\partial}{\partial s}w_k(t,s) - \frac{k^2}{s} w_k(t,s) \right ) \ds \\  = -2\pi r_0^{-k}\left(r_0 \frac{\partial w_k(t,r)}{\partial r}\Big|_{r=r_0} + k w_k(t,r_0) \right )=0,
\end{align}
where $\Delta_k w(t,r)$ is defined in (\ref{fourierlaplace}).

In a similar way using complex conjugation of moment relations (\ref{noslipcondintegral2}) we obtain boundary condition for $k<0$:
\begin{align*}
\int_{B_{r_0}} {\Delta w}\overline{z^{k}}\dx  = \int_{r_0}^\infty \int_0^{2\pi} \frac {\Delta w}{s^{-k}e^{ik\phi}}sdsd\varphi = 2\pi \int_{r_0}^\infty s^{k+1}\Delta_k w_k(s) ds \\= -2\pi r_0^{k}\left(r_0 \frac{\partial w_k(t,r)}{\partial r}\Big|_{r=r_0} - k w_k(t,r_0) \right )=0.
\end{align*}

\subsection{Helmholtz equation in the exterior of the disc}

Now we are ready to derive the boundary condition to the vorticity equation. In fact it will be an integral condition. But its first approximation will be the same boundary condition (\ref{robin_bound}) as for Stokes flow.

Consider Helmholtz equation for vorticity
\begin{align} \label{helmholtzeq}
\frac{\partial w(t,\bx)}{\partial t}	 - \Delta w(t,\bx) + (\bv,\nabla w) = 0,
\end{align}
where $\bv$ is redtored from $w$ via Biot-Savar law (\ref{BSformula}).

\begin{thm}Given initial datum $\bv_0(\bx)$, $\curl \bv_0\in L_1(B_{r_0})$ satisfying no-slip condition (\ref{bound}), infinity condition (\ref{boundinf}), zero-circularity (\ref{zerocirculation}) and $\bv(t,\bx)$ be the solution of (\ref{maineqns})-(\ref{boundinfns}) in $B_{r_0}$. Then $w(t,\bx)=\curl \bv(t,\bx)$ satisfies 
\begin{align}
&r_0 \frac{\partial w_k(t,r)}{\partial r}\Big|_{r=r_0} +  |k| w_k(t,r_0)  = \nonumber \\
&~~~~\left \{ {  \begin{matrix}
\frac {|k| r_0^{|k|}}{2\pi}
\int_{B_{r_0}} (\bv^\CM_\infty - \bv^\CM){z^{-|k|-1}} w(t,\bx)\dx,~k\geq 0, \\
\frac {|k| r_0^{|k|}}{2\pi}
\int_{B_{r_0}} \overline{(\bv^\CM_\infty - \bv^\CM)z^{-|k|-1}} w(t,\bx)\dx,~k<0.
\end{matrix} } \right . \label{robin_bound_ns}
\end{align}
\end{thm}

\begin{proof}
Fix $k>0$ and divide this equation by $z^k$ and integrate over the exterior of the disc $B_{r_0}$. From moment relations (\ref{noslipcondintegral2}) follows
$$
\frac d{dt}\int_{B_{r_0}} \frac {w(t,\bx)}{z^k}  \dx = 0
$$

Denote $$\bv^\CM = v_1+iv_2,~\bv^\CM_\infty = v_{1,\infty}+iv_{2,\infty}.$$
 
Then
\begin{align*}
\int_{B_{r_0}} \frac {(\bv,\nabla w)}{z^k} \dx = -\int_{B_{r_0}} (\bv,\nabla z^{-k})w(t,\bx) \dx \\
=k \int_{B_{r_0}} \frac{v_1+iv_2}{z^{k+1}} w(t,\bx) \dx=k \int_{B_{r_0}} \frac{\bv^\CM}{z^{k+1}} w(t,\bx)\dx
\end{align*}

From (\ref{noslipcondintegral2})
$$
k \int_{B_{r_0}} \frac{\bv^\CM_\infty}{z^{k+1}} w(t,\bx)\dx = 0
$$
and
$$
\int_{B_{r_0}} \frac {(\bv,\nabla w)}{z^k} \dx = k \int_{B_{r_0}} \frac{\bv^\CM-\bv^\CM_\infty}{z^{k+1}} w(t,\bx)\dx.
$$

Using (\ref{integrationbypart_dissip}) we get
$$
\int_{B_{r_0}}  \frac {\Delta w}{z^{k}}\dx=-2\pi r_0^{-k}\left(r_0 \frac{\partial w_k(t,r)}{\partial r}\Big|_{r=r_0} + k w_k(t,r_0) \right ).
$$

Then for $k\geq 0$
\begin{align*}
\left(r_0 \frac{\partial w_k(t,r)}{\partial r}\Big|_{r=r_0} + k w_k(t,r_0) \right ) = \frac {k r_0^k}{2\pi}
\int_{B_{r_0}} \frac{\bv^\CM_\infty - \bv^\CM}{z^{k+1}} w(t,\bx)\dx.
\end{align*}

Since $w_{-k}(t,r) = \overline{w_k(t,r)}$ then (\ref{robin_bound_ns}) holds for $k<0$. Theorem is proved.

\end{proof}

%
%

%
%

\begin{remark}
If $\|\bv - \bv_\infty\|$ is small in some integral norm for a well-streamlined body , then the right side in (\ref{robin_bound_ns}) transfers to boundary condition (\ref{robin_bound}). So, (\ref{robin_bound}) becomes a rather accurate approximation for Navier-Stokes system. From this fact naturally occurs boundary control problem with unknown function $u_k(t)$:
\begin{equation}\label{robin_bound_control}
r_0\frac{\partial w_k(t,r)}{\partial r}\Big|_{r=r_0} + |k| w_k(t,r_0) = u_k(t),~k \in \ZM,
\end{equation}
where feedback controls $u_k(t)$ are treated as a small correction of zero boundary condition.
\end{remark}

\subsection{Linear vorticity equation in a simply connected domains}Here we find out that Robin-type boundary (\ref{robin_bound}) works well as the no-slip boundary condition in exterior domains. It keeps moment relations (\ref{noslipcondintegralrieman}) under the Stokes flow, and thus $M$ is an invariant affine subspace.

We impose additional requirement on $\Omega$ that $\Phi^{-1}$ can be represented by absolutely convergent series
\begin{equation} \label{phireq}
\Phi^{-1}(z) = z+\sum\limits_{n=1}^\infty \frac {b_n}{z^n}.
\end{equation} 

\begin{thm}\label{stokesomegathm}Let $\Omega=\RM^2 \setminus B$, where $B$ is a bounded simple-connected domain with smooth boundary and $\Phi$ be a Riemann mapping from $\Omega$ into exterior of the disc satisfying (\ref{phireq}), $\bv_0(\bx)$ - is the initial datum, satisfying no-slip condition (\ref{bound}), infinity condition (\ref{boundinf}), zero-circularity (\ref{zerocirculation}), $\curl \bv_0\in L_1(B_{r_0})$, and $\bv(t,\bx)$ be the solution of (\ref{maineqns})-(\ref{boundinfns}) in $\Omega$. Then $w(t,\bx)$=$\curl \bv(t,\Phi^{-1}(\bx))$ satisfies (\ref{robin_bound}).
\end{thm}

\begin{proof}
After Riemann mapping Stokes equations reduce to scalar equation on vorticity
$$
|(\Phi^{-1})'(z)|^2\partial_t w(t,x) - \Delta w=0.
$$

Fix $k>0$ and divide this equation by $z^k=(x_1+ix_2)^k$. Then integrate it over the exterior of the disc $B_{r_0}$. From moment relations (\ref{noslipcondintegral2_omega}) follows
$$
\frac d{dt}\int_{B_{r_0}} \frac {\overline{(\Phi^{-1})'(z)}}{z^k} w(t,\bx) \dx = 0.
$$ 

Using $|(\Phi^{-1})'(z)|^2=(\Phi^{-1})'(z)\overline{(\Phi^{-1})'(z)}$ and
$$
(\Phi^{-1})'(z) = \sum\limits_{n=0}^\infty \frac {c_n}{z^n}
$$
with some coefficients $c_n$ we have
\begin{align*}
\frac d{dt}\int_{B_{r_0}} \frac{|(\Phi^{-1})'(z)|^2 }{z^k} w(t,x) \dx = \frac d{dt}\int_{B_{r_0}}  \sum\limits_{n=0}^\infty \frac {c_n \overline{(\Phi^{-1})'(z)}}{z^{n+k}}w(t,\bx) \dx\\=\sum\limits_{n=0}^\infty c_n \frac d{dt}\int_{B_{r_0}}   \frac { \overline{(\Phi^{-1})'(z)}}{z^{n+k}}w(t,\bx) \dx=0,
\end{align*}
and thus
$$
\int_{B_{r_0}}  \frac {\Delta w}{z^{k}}\dx=0.
$$
In other hand
\begin{align*}
\int_{B_{r_0}} \frac {\Delta w}{z^{k}}\dx = \int_{r_0}^\infty \int_0^{2\pi} \frac {\Delta w}{s^ke^{ik\phi}}sdsd\varphi = 2\pi \int_{r_0}^\infty s^{-k+1}\Delta_k w_k(s) ds \\= -2\pi r_0^{-k}\left(r_0 \frac{\partial w_k(t,r)}{\partial r}\Big|_{r=r_0} + k w_k(t,r_0) \right )=0.
\end{align*}

Using $w_{-k}(t,r) = \overline{w_k(t,r)}$ we will have condition (\ref{robin_bound}) for $k<0$.

\end{proof}

\subsection{Helmholtz equation in a simply connected domains}
For the linear vorticity equation we found the boundary condition. Here we find it for nonlinear equation. Vorticity equation (\ref{maineqw})  corresponding to Navier-Stokes system after Riemann mapping reduces to 
\begin{align}\label{helmholtzeq2}
|(\Phi^{-1})'(z)|^2\partial_t w(t,\bx) - \Delta w + B(v,w) =0,
\end{align}
where
$$
B(v,w) = \real (\Phi^{-1})' (\bv,\nabla w) - \imag (\Phi^{-1})' (\bv^\perp,\nabla w).
$$
with $\bx \in B_{r_0}$.

Supply this equation with
\begin{align}
&r_0 \frac{\partial w_k(t,r)}{\partial r}\Big|_{r=r_0} +  |k| w_k(t,r_0)  = \nonumber \\
&~~~~\left \{ {  \begin{matrix}
\frac {|k| r_0^{|k|}}{2\pi}
\int_{B_{r_0}} (\bv^\CM_\infty - \bv^\CM){z^{-|k|-1}} \overline{(\Phi^{-1})'} w(t,\bx)\dx,~k\geq 0, \\
\frac {|k| r_0^{|k|}}{2\pi}
\int_{B_{r_0}} \overline{(\bv^\CM_\infty - \bv^\CM)z^{-|k|-1}} (\Phi^{-1})' w(t,\bx)\dx,~k<0.
\end{matrix} } \right . \label{robin_bound_ns_general}
\end{align}


\begin{thm}Let $\Omega$, $\Phi$ as in Theorem  \ref{stokesomegathm}, given initial datum $\bv_0(\bx)$, $\curl \bv_0\in L_1(\Omega)$ satisfying no-slip condition (\ref{bound}), infinity condition (\ref{boundinf}), zero-\\circularity (\ref{zerocirculation}) and $\bv(t,\bx)$ be the solution of (\ref{maineqns})-(\ref{boundinfns}) in $\Omega$. Then vorticity $w(t,\cdot)=$ $\curl \bv(t,\Phi^{-1}(z))$ satisfies (\ref{robin_bound_ns_general}).
\end{thm}

\begin{proof}
Apply the same steps as in previous subsections: for fixed $k>0$ divide this equation by $z^k$ and integrate over $B_{r_0}$. In the previous subsection we obtained
$$
\frac d{dt}\int_{B_{r_0}} \frac {\overline{(\Phi^{-1})'(z)}}{z^k} w(t,\bx) \dx = 0,
$$ 
%
and
\begin{align*}
\int_{B_{r_0}} \frac {\Delta w}{z^{k}}\dx = -2\pi r_0^{-k}\left(r_0 \frac{\partial w_k(t,r)}{\partial r}\Big|_{r=r_0} + k w_k(t,r_0) \right ).
\end{align*}

From (\ref{freediv3}), (\ref{curleq3})
$$
 \int_{B_{r_0}} \frac {\real (\Phi^{-1})'}{z^{k}} w \divv \bv \dx  = -
 \int_{B_{r_0}} \frac {\imag (\Phi^{-1})'}{z^{k}} w \curl \bv   \dx.  
$$

Then using Cauchy-Riemann equations
\begin{align*}
\int_{B_{r_0}} \frac {B(v,w)}{z^{k}}\dx = \int_{B_{r_0}} \frac {\real (\Phi^{-1})'}{z^{k}}(\bv,\nabla w) -
 \frac {\imag (\Phi^{-1})'}{z^{k}} (\bv^\perp,\nabla w) \dx\\=
 \int_{B_{r_0}} \left ( -\frac {\real (\Phi^{-1})'}{z^{k}} w \divv \bv  -
 \frac {\imag (\Phi^{-1})'}{z^{k}} w \curl \bv \right )  \dx  \\
 -\int_{B_{r_0}} w \left ( \left ( v, \nabla \frac {\real (\Phi^{-1})'}{z^{k}} \right ) - \left ( v^\perp, \nabla \frac {\imag (\Phi^{-1})'}{z^{k}} \right ) \right ) \dx \\ =
 k\int_{B_{r_0}} \frac {\overline{(\Phi^{-1})'}}{z^{k+1}} (v_1+iv_2) w(t,\bx) \dx
 =k\int_{B_{r_0}} \frac {\overline{(\Phi^{-1})'}}{z^{k+1}} \bv^\CM w(t,\bx)   \dx
\end{align*}

Since from (\ref{noslipcondintegral2_omega})
$$
k\int_{B_{r_0}} \frac {\overline{(\Phi^{-1})'}}{z^{k+1}} \bv^\CM_\infty w(t,\bx)   \dx = 0, 
$$
then
$$
\int_{B_{r_0}} \frac {B(v,w)}{z^{k}}\dx = k\int_{B_{r_0}} \frac {\overline{(\Phi^{-1})'}}{z^{k+1}} (\bv^\CM-\bv^\CM_\infty ) w(t,\bx)   \dx,
$$
from which follows (\ref{robin_bound_ns_general}) for $k\geq 0$.

For $k<0$ (\ref{robin_bound_ns_general}) follows from the identity $w_{-k}(t,r) = \overline{w_k(t,r)}$. Theorem is proved.

\end{proof}

This theorem says, that for Navier-Stokes system if $\|\bv^\CM-\bv^\CM_\infty\|$ is small, then (\ref{robin_bound}) works well. Nevertheless this boundary condition requires some corrections like boundary control (\ref{robin_bound_control}) in order to stay on invariant affine subspace $M$. 

\section{No-slip boundary condition for vorticity in exterior domains}

In this section we construct boundary condition for vorticity which approximates no-slip condition. With help of Riemann mapping we had reduced the Helmholtz equation (\ref{maineqw}) defined in $\Omega$ to (\ref{helmholtzeq2}) defined in $B_{r_0}$. Formulas (\ref{robin_bound_ns_general}) present integral condition which approximately equal to boundary condition (\ref{robin_bound}) for well streamlined obstacle when $\bv \simeq \bv_\infty$. In this section using fixed point theorem we will derive boundary condition for vorticity in the form like (\ref{_3:bound:nonlin}) which ensure the solution to satisfy $w_N(t,\cdot) \in M_N$ defined in (\ref{noslipcondintegral2_omega_N}).
Then the velocity $\bv(t,\bx)$ restored from $w(t,\bx)$ via Biot-Savar law (\ref{BSformula}) will be the solution of Navier-Stokes system with approximate no-slip boundary condition satisfying  (\ref{weakconv}).

Fix $N\in \NM$. We supply vorticity equation (\ref{helmholtzeq2}) with
\begin{align}
&w(0,\bx)=w_0(\bx), \label{_3:initw}\\
&w(t,\bx) \to 0,~|\bx|\to \infty, \label{_3:boundinfw}
\end{align}
and boundary condition 
\begin{equation}\label{_3:bound1}
r_0\frac{\partial w_k(t,r)}{\partial r}\Big|_{r=r_0} + k w_k(t,r_0) = u_k(t),~k = 0,1,\dots, N,
\end{equation}
where $\vec u(t) = \{u_k(t)\}_{k=0}^N$ is the set of unknown boundary functions.
Then since $w_{-k}(t,r) = \overline{w_k}(t,r)$ for $k=-N, -N-1, \dots, -1$
\begin{equation}\label{_3:bound2}
r_0\frac{\partial w_k(t,r)}{\partial r}\Big|_{r=r_0} - k w_k(t,r_0) = \overline{u_k(t)}.
\end{equation}

For $|k|>N$ set 
\begin{equation}\label{_3:bound3}
r_0\frac{\partial w_k(t,r)}{\partial r}\Big|_{r=r_0} + k w_k(t,r_0) = 0.
\end{equation}

We construct approximate boundary condition for Helmholtz equation under the assumption of unique resolvability of (\ref{helmholtzeq2}), (\ref{_3:initw})-(\ref{_3:bound3}). For the exterior of the disc the solvability theorems are proved in Appendixes \ref{oseen_sect}, \ref{ns_sect}.

For the Helmholtz problem we should use  Sobolev space $H^1(B_{r_0})$ as the phase space for $w(t,\cdot)$.  Difficulty lies in the fact that the velocity field $\bv$ isn't well defined in $H^1(B_{r_0})$.
In order to obtain $\bv(t,\cdot) \in L_\infty(B_{r_0})$ from Lemma \ref{lembsest} we need involve $L_1$ into phase space for $w_{-1}, w_0, w_1$. So, we will use phase space
$$
W = \left \{w \in H^1(B_{r_0}),~w_{-1}, w_0, w_1 \in L_1(r_0,\infty;r) \right \}.
$$

\begin{thm} Suppose that for fixed $N>0$, $T>0$ there exists the solution $w_N(t,\bx) \in C \left ( [0,T], W) \right )$ of Helmholtz equation (\ref{helmholtzeq2}), (\ref{_3:initw})-(\ref{_3:bound3}) which is locally Lipschitz continuous due to $u_k \in C[0,T]$, $k=0,\dots, N$, $w_0(\bx) \in H^1(B_{r_0})$. Then for some $M>0$ and any initial datum with $\|w_0(\cdot)\|_{H^1}\leq M$ there exists $\vec u(t) = \{u_k(t)\}_{k=0}^N$, that $w_N(t,\cdot) \in M_N$ for all $t \in [0,T]$.
\end{thm}

\begin{proof} 
The relation (\ref{robin_bound_ns_general}) can be rewritten as
\begin{align}\label{boundary_equation}
&\vec u(t) = F(\vec u) 
\end{align}
with the mapping $F(\vec u) : \left (C[0,T]\right )^{2N+1} \to \left (C[0,T]\right )^{2N+1}$:
\begin{align*}
&F[\vec u(\cdot)] = \nonumber \\
&~~~~\left \{ {  \begin{matrix}
\frac {|k| r_0^{|k|}}{2\pi}
\int_{B_{r_0}} (\bv^\CM_\infty - \bv^\CM){z^{-|k|-1}} \overline{(\Phi^{-1})'} w(t,\bx)\dx,~k = 0, \dots, N \\
\frac {|k| r_0^{|k|}}{2\pi}
\int_{B_{r_0}} \overline{(\bv^\CM_\infty - \bv^\CM)z^{-|k|-1}} (\Phi^{-1})' w(t,\bx)\dx,~k = -N, \dots, -1.
\end{matrix} } \right .  
\end{align*}
where $w(t,\bx)$ is the solution of boundary-value problem (\ref{helmholtzeq2}), (\ref{_3:initw})-(\ref{_3:bound3}) with boundary condition $\vec u(t)$. $F$ is well-defined since $k/z^{|k|+1} \in L_2(B_{r_0})$ and in virtue of Lemma \ref{lembsest} from 
\begin{align*}
\left \| k (\bv^\CM_\infty - \bv^\CM(t,\cdot) ){z^{-|k|-1}}  \right \|_{L_2(B_{r_0})}  \leq C \|w(t,\bx)\|_{W}
\end{align*}
follows
\begin{align*}
\left |F[\vec u(t)] \right | \leq C \|w(t,\cdot)\|^2_{W}
\end{align*}
with some new constant $C>0$. From $w(t,\bx) \in C\left([0,T];H^1(B_{r_0}) \right)$ follows $F[\vec u(t)] \in \left (C[0,T]\right )^{2N+1}$.

Fix $M>0$ and
take $\vec u_1(\cdot)$, $\vec u_2(\cdot)$ from $\left (C[0,T]\right )^{2N+1}$ with $\|\vec u_i\|_{L_\infty}<M$, $i=1,2$. Let $w_1$, $w_2$ - be the corresponding solutions of (\ref{helmholtzeq2}), (\ref{_3:initw})-(\ref{_3:bound3}) with $w_0(\bx)$ satisfying $\|w_0(\bx)\|\leq M$. 

$F$ is the contraction mapping with respect to $\vec u$.  Indeed
\begin{align*}
&\left | F[\vec u_1(\cdot)] - F[\vec u_2(\cdot)] \right | \leq \\&~~~~~~~~
\frac {|k| r_0^{|k|}}{2\pi} 
\int_{B_{r_0}} \Big |  (\bv^\CM_\infty - \bv_1^\CM){z^{-|k|-1}} \overline{(\Phi^{-1})'} w_1(t,\bx)  - \\&~~~~~~~~
(\bv^\CM_\infty - \bv_2^\CM){z^{-|k|-1}} \overline{(\Phi^{-1})'} w_2(t,\bx)\Big | \dx \leq \\&~~~~~~~~ \frac {r_0^{|k|}}{2\pi} 
\int_{B_{r_0}} \left |  k(\bv^\CM_\infty - \bv_1^\CM){z^{-|k|-1}}  (w_1(t,\bx) - w_2(t,\bx)) \right | \dx  +  \\&~~~~~~~~ \frac {r_0^{|k|}}{2\pi} 
\int_{B_{r_0}} \Big |  k(\bv_2^\CM - \bv_1^\CM){z^{-|k|-1}}  w_2(t,\bx) \Big | \dx .
\end{align*}

From Lemma \ref{lembsest} again we have with some $C>0$
\begin{align*}
\left \| k \left (\bv_1^\CM(t,\cdot) - \bv_2^\CM(t,\cdot)\right ){z^{-|k|-1}}  \right \|_{L_2(B_{r_0})}  \leq C \|w_1(t,\bx) - w_2(t,\bx)\|_{W}
\end{align*}
and
$$
\| k(\bv^\CM_\infty - \bv_1^\CM(t,\cdot)){z^{-|k|-1}}\|_{L_2(B_{r_0})}  
\leq C\| w_1(t,\bx) \|_{W}.
$$
Note, that from Lipschitz continuity of $w$ due to $\vec u$ the norms $\| w_1(t,\bx) \|_{W}$, $\| w_2(t,\bx) \|_{W}$ can be arbitrary small for small $M>0$. Then the estimate
\begin{align*}
\left | F[\vec u_1(\cdot)] - F[\vec u_2(\cdot)] \right | \leq
 C(M) \|w_1(t,\cdot) - w_2(t,\cdot)\|_{W} 
\end{align*}
holds with $C(M) \to 0$ as $M \to 0$.

Since with some $L>0$
$$
\|w_1 - w_2\|_{C \left ( [0,T], W) \right )} \leq L \| u_1 - u_2 \|_{\left (C[0,T]\right )^{2N+1}}
$$
then for small $M>0$ we will finally have
$$
\left \| F[\vec u_1(\cdot)] - F[\vec u_2(\cdot)] \right \|_{\left (C[0,T]\right )^{2N+1}} \leq K \|\vec u_1(\cdot) - \vec u_2(\cdot)\|_{\left (C[0,T]\right )^{2N+1}},
$$
with $K<1$ and from the fixed-point theorem there exists the solution of (\ref{boundary_equation}). The theorem is proved.
\end{proof}

\begin{remark} If  $w_0(\bx) \in H^2(B_{r_0})$ then the solution will satisfy $w(t,\cdot)\in H^2(B_{r_0})$ and the velocity on the boundary will tend to zero in $L_2(\partial \Omega)$ as $N\to\infty$ according to (\ref{strongconv}).
If we further increase the smoothness of the initial datum then we can obtain the uniform convergence of velocity to no-slip condition on the boundary.
\end{remark}

\section{Appendix}
Here we study semigroup for Stokes flow and prove uniqueness solvability for Oseen and Helmholtz equations using fixed-point argument.

\subsection{Stokes semi-group estimates}
The boundary-value problem (\ref{stokeseq}), (\ref{robin_bound}), (\ref{boundinfvorticity})  can be solved using special  Weber-Orr transform (see \cite{AG} for more details):
\begin{equation}\label{int:weberorr}
W_{k,l}[f](\lambda) = \int_{r_0}^\infty \frac{J_{k}(\lambda s)Y_{l}(\lambda r_0) - Y_{k}(\lambda s)J_{l}(\lambda r_0)}{\sqrt{J_{l}^2(\lambda r_0) + Y_{l}^2(\lambda r_0)}} f(s) s \ds,~k\in \ZM,
\end{equation}
where
$J_{k}(r)$, $Y_{k}(r)$ - are the Bessel functions of the first and second type (see \cite{BE}).

\vskip 5pt
The inverse transform is defined by the formula
\begin{equation}\label{int:weberorrinv}
W^{-1}_{k,l}[\hat f](r) = \int_{0}^\infty \frac{J_{k}(\lambda r)Y_{l}(\lambda r_0) - Y_{k}(\lambda r)J_{l}(\lambda r_0)}{\sqrt{J_{l}^2(\lambda r_0) + Y_{l}^2(\lambda r_0)}} \hat f (\lambda) \lambda \dlambda.
\end{equation}

The case of $W_{|k|,|k|-1}$  is of special interest and this transform can be considered as hydrodynamical. In \cite{AG} was proved the following
\begin{thm} Let vector field $\bv_0(\bx)$ satisfies (\ref{freediv}), (\ref{bound}), (\ref{boundinf}), (\ref{zerocirculation}), the vorticity $\curl \bv_0(\bx)$ $ \in L_1(B_{r_0})$, and its Fourier series as well as Fourier series for its vorticity $w_0(\bx)$ with coefficients $w_k^0(r)$ converges, $\bv (t, \bx)$ be the solution of (\ref{stokeseqvelocity}), (\ref{freedivns})-(\ref{boundinfns}). Then $w(t,\bx)=\curl \bv (t, \bx)$ satisfies equation (\ref{stokeseq}), boundary conditions (\ref{robin_bound}),  (\ref{boundinfvorticity}) and is given via Fourier coefficients:
\begin{align}\label{MyGreatFormula}
w_k(t,r) = W^{-1}_{|k|,|k|-1} \left [ e^{-\lambda^2 t} W_{|k|,|k|-1} [w^0_k(\cdot)](\lambda) \right ](t,r).
\end{align}

\end{thm}
This formula gives an explicit form of the solution for Stokes system in the exterior of the disc in terms of vorticity. From Biot-Savar law one can get the velocity field. 

Generalised Weber-Orr transform has a non-trivial kernel. So, the transform (\ref{int:weberorrinv}) is the inverse one to (\ref{int:weberorr}) only for functions, which are orthogonal to the kernel. But from no-slip condition (\ref{noslipcondintegral}) follows orthogonality of Fourier coefficients $w_k(t,\cdot)$ to the kernel of $W_{|k|,|k|-1}$. The invertibility   of the generalised Weber-Orr transform was studied in detail in \cite{AG}. 

So, Weber-Orr transforms satisfy Bessel-type inequality instead of the Plancherel equity:
\begin{align}\label{besselineq}
\| W_{k,l}[f] \|_{L_2(0,\infty; \lambda) }^2 \leq \|f\|_{L_2(r_0,\infty; r)}^2. 
\end{align}

From differentiation rules for Bessel functions follows
\begin{align}
\frac{\partial}{\partial r} W^{-1}_{k,k-1} [f]  = \frac 1 2  \left ( W^{-1}_{k-1,k-1}[\lambda f] - W^{-1}_{k+1,k-1}[\lambda f] \right ), \label{diffweber1}\\
\frac kr W^{-1}_{k,k-1}[ f] =  \frac 12 \left ( W^{-1}_{k+1,k-1}[\lambda  f] + W^{-1}_{k-1,k-1}[\lambda  f] \right ).\label{diffweber2}
\end{align}

In polar coordinates with unit vectors $\mathbf{e_r}$, $\mathbf{e_\varphi}$ the gradient is defined as $\nabla = \frac{\partial}{\partial r} \mathbf{e_r} + \frac 1r \frac{\partial}{\partial \varphi} \mathbf{e_\varphi}$ with polar coordinates $\nabla_r$, $\nabla_\varphi$. Multiplier $\frac kr$ corresponds to differentiation with respect to $\varphi$ of function $f(r)e^{ik\varphi}$.

Formula (\ref{diffweber2}) involves Weber transform $W^{-1}_{k+1,k-1}$ which also possesses non-trivial kernel and satisfy Bessel inequality
$$
\| W_{k+1,k-1}[f] \|_{L_2(0,\infty; \lambda) }^2 \leq \|f\|_{L_2(r_0,\infty; r)}^2. 
$$

And vice versa, multiplication by $\lambda$ transfers to differentiation but in more general sense including not only $\frac{\partial}{\partial r}$ but also angle derivative $\frac{\partial}{\partial \varphi}$ expressed by multiplier $\frac kr$. Indeed, using
\begin{align*}
\lambda J_k(\lambda r) = \frac{k-1}r J_{k-1}(\lambda r) - \lambda J'_{k-1}(\lambda r), \\
\lambda Y_k(\lambda r) = \frac{k-1}r Y_{k-1}(\lambda r) - \lambda Y'_{k-1}(\lambda r)
\end{align*}
we will have
\begin{equation} \label{diffweber3}
\lambda W_{k,k-1} [f] = W_{k-1,k-1} [\frac {k f}r] + W_{k-1,k-1} [f'(\cdot )].
\end{equation}

Formula (\ref{MyGreatFormula}) in theorem defines the Stokes semigroup $S(t)$ which corresponds to problem  (\ref{stokeseq}),  (\ref{robin_bound}),  (\ref{boundinfvorticity}). The estimates of this semi-group are given by the following
\begin{prop} \label{stokes_semigroup_thm}For $t > 0$ Stokes semigroup $S(t)$ for vorticity $\curl$ satisfies
\begin{align*}
\|S(t)w_0\|_{L_2(B_{r_0})} \leq \|w_0\|_{L_2(B_{r_0})}, \\ 
\|\nabla S(t)w_0\|_{L_2(B_{r_0})} \leq \frac 1{\sqrt {et} }\|w_0\|_{L_2(B_{r_0})},\\
\|S(t)w_0\|_{H^1(B_{r_0})} \leq \sqrt 3 \|w_0\|_{H^1(B_{r_0})}.
\end{align*}
\end{prop}

\begin{proof}
From Bessel inequality we have
\begin{align*}
\|w_k(t,r)\|^2_{L_2(r_0,\infty, r)} \leq \| e^{-\lambda^2 t} W_{|k|,|k|-1} [w^0_k(\cdot)](\lambda)\|^2_{L_2(0,\infty, \lambda)} \\
\leq \| W_{|k|,|k|-1} [w^0_k(\cdot)](\lambda)\|^2_{L_2(0,\infty, \lambda)} \leq 
 \| w^0_k(r)\|_{L_2(r_0,\infty, r )}.
\end{align*}
Summarizing by $k$ we obtain the first estimate.

Fix $k \geq 0$. Then from estimate
$$
 |\lambda e^{-\lambda^2 t}| \leq \frac 1{\sqrt {2et} } 
$$
and Bessel inequality (\ref{besselineq}) from (\ref{diffweber1}) we have 
\begin{align*}
\left \|\frac{\partial}{\partial r} w_k(t,r) \right \|_{L_2(r_0,\infty, r)} \leq  \| \lambda e^{-\lambda^2 t} W_{k,k-1} [w^0_k(\cdot)](\lambda)\|_{L_2(0,\infty, \lambda)} \\ \leq  
\frac {\| w^0_k(r)\|_{L_2(r_0,\infty, r )}}{\sqrt {2et}}
\end{align*}
and
$$
\|\nabla_r S(t)w_0\|_{L_2(B_{r_0})} \leq \frac 1{\sqrt {2et} }\|w_0\|_{L_2(B_{r_0})}.
$$

From (\ref{diffweber2}) in a similar way follows
\begin{align*}
\left \|\frac kr w_k(t,r) \right \|_{L_2(r_0,\infty, r)} \leq \frac {\| w^0_k(r)\|_{L_2(r_0,\infty, r )}}{\sqrt {2et}}.
\end{align*}

Then we have
$$
\|\nabla_\varphi S(t)w_0\|_{L_2(B_{r_0})} \leq \frac 1{\sqrt {2et} }\|w_0\|_{L_2(B_{r_0})}
$$
and the second estimate of the proposition is proved.

Now we will prove the last inequality. Using (\ref{diffweber1}),(\ref{diffweber3})
\begin{align*}
\left \|\frac{\partial}{\partial r} w_k(t,r) \right \|_{L_2(r_0,\infty, r)} \leq  \| \lambda e^{-\lambda^2 t} W_{k,k-1} [w^0_k(\cdot)](\lambda)\|_{L_2(0,\infty, \lambda)} \\ = \left \| e^{-\lambda^2 t}  W_{k-1,k-1} [\frac {k w^0_k}r +  \frac{\partial}{\partial r} w^0_k(\cdot )]  \right \| \leq \left \|\frac {k w^0_k}r \right \| + \left \| \frac{\partial}{\partial r} w^0_k(\cdot ) \right \|
\end{align*}
and so
\begin{align} \label{nablarsemigroup}
\|\nabla_r S(t)w_0\|_{L_2(B_{r_0})} \leq \sqrt 2 \|\nabla w_0\|_{L_2(B_{r_0})}.
\end{align}

In a similar way from (\ref{diffweber2}), (\ref{diffweber3})
\begin{align*}
\left \|\frac kr w_k(t,r) \right \|_{L_2(r_0,\infty, r)} \leq   \| \lambda e^{-\lambda^2 t}  W_{k,k-1} 
 [w^0_k(\cdot)](\lambda)\|_{L_2(0,\infty, \lambda)}
  \\  \leq
  \left (  \left \|\frac {k w^0_k}r \right \| +   \left \| \frac{\partial}{\partial r} w^0_k(\cdot ) \right \| \right )
\end{align*}
and finally
$$
\|\nabla_\varphi S(t)w_0\|_{L_2(B_{r_0})} \leq \sqrt 2 \|\nabla w_0\|_{L_2(B_{r_0})}
$$
combined with (\ref{nablarsemigroup}) and the first estimate of the proposition gives the last estimate.

\end{proof}
\subsection{Existence theorem for Oseen equation in exterior of the disc} \label{oseen_sect}
For a fixed irrotational velocity field $\tilde \bv(t,\bx) $ consider Oseen equation
\begin{equation} 
\frac{\partial w(t,\bx)}{\partial t}	 - \Delta w + (\tilde \bv,\nabla)w  = 0,  \label{oseeneqw} 
\end{equation} 
with initial datum (\ref{_3:initw}) 
and condition at infinity (\ref{_3:boundinfw}).

We supply the problem (\ref{oseeneqw}), (\ref{_3:initw}), (\ref{_3:boundinfw}) with boundary conditions (\ref{_3:bound1})-(\ref{_3:bound3}).

Define for $k, l\in \ZM$
$$
R_{k,l}(\lambda, r) = \frac{J_{k}(\lambda r)Y_{l}(\lambda r_0) - Y_{k}(\lambda r)J_{l}(\lambda r_0)}{\sqrt{J_{l}^2(\lambda r_0) + Y_{l}^2(\lambda r_0)}}.
$$

From properties of Bessel functions
\begin{align*}
J_{k - 1}(r)={\frac {k }{r}}J_{k }(r) + J'_{k }(r) \\
Y_{k - 1}(r)={\frac {k }{r}}Y_{k }(r) + Y'_{k }(r)
\end{align*}
follows
$$
k {R_{k,k-1}(\lambda, r_0)} + \lambda r_0   {R'_{k,k-1}(\lambda, r_0)} =  r_0   {R_{k-1,k-1}(\lambda, r_0)} = 0.
$$

Also holds (\cite{BE}):
$$
R_{k,k-1}(\lambda, r_0) = \frac{J_{k}(\lambda r_0)Y_{k-1}(\lambda r_0) - Y_{k}(\lambda r_0)J_{k-1}(\lambda r_0)}{\sqrt{J_{k-1}^2(\lambda r_0) + Y_{k-1}^2(\lambda r_0)}} = \frac 2{\pi r_0 \lambda }.
$$

We apply Weber-Orr transform $W_{k,k-1}$ (\ref{int:weberorr}) to Helmholtz equation for Oseen flow (\ref{oseeneqw}). First we find how it acts on Laplace operator. Then using integration by parts we will have 
\begin{align*}
&W_{k,k-1}[\Delta_k w_k(t, r)] = - r_0 \frac{\partial w_k(t, r)}{\partial r} \frac{R_{k,k-1}(\lambda, r)}{\sqrt{J_{k-1}^2(\lambda, r_0) + Y_{k-1}^2(\lambda, r_0)}} \Bigg |_{r_0}^\infty \\ & +  r_0 w_k(t, r) \frac{\lambda R'_{k,k-1}(\lambda, r)}{\sqrt{J_{k-1}^2(\lambda r_0) + Y_{k-1}^2(\lambda r_0)}} \Bigg |_{r_0}^\infty - \lambda^2 W_{k,k-1}[w_k(t, r)] \\ &
=\left (u_k(t) - k w_k(t,r_0)  \right) \frac{  R_{k,k-1}(\lambda, r_0)}{\sqrt{J_{k-1}^2(\lambda r_0) + Y_{k-1}^2(\lambda r_0)}}  \\ & - r_0 w_k(t, r_0) \frac{\lambda R'_{k,k-1}(\lambda, r)}{\sqrt{J_{k-1}^2(\lambda r_0) + Y_{k-1}^2(\lambda r_0)}} - \lambda^2 \hat w_k(t, \lambda)\\ &
= - \frac{w_k(t,r_0)}{\sqrt{J_{k-1}^2(\lambda r_0) + Y_{k-1}^2(\lambda r_0)}} \left (k {R_{k,k-1}(\lambda, r_0)} + \lambda r_0   {R'_{k,k-1}(\lambda, r_0)}  \right ) \\ &- \lambda^2 \hat w_k(t, \lambda) + u_k(t)\frac{  R_{k,k-1}(\lambda, r_0)}{\sqrt{J_{k-1}^2(\lambda r_0) + Y_{k-1}^2(\lambda r_0)}} \\ & =
   \frac{2 u_k(t)}{\pi r_0 \lambda \sqrt{J_{k-1}^2(\lambda r_0) + Y_{k-1}^2(\lambda r_0)}}  - \lambda^2 \hat w_k(t, \lambda).
\end{align*}

Then the Helmholtz equation (\ref{oseeneqw}) in terms of Fourier coefficients $w_k$ after Weber-Orr transform can be written as
$$
\partial \hat w_k(t, \lambda) + \lambda^2 \hat w_k(t, \lambda) + r_k(\lambda) u_k(t)  + W_{k,k-1}[(\bv,\nabla)w]_k = 0
$$
where
$$
r_k(\lambda) = \frac{2}{\pi r_0 \lambda \sqrt{J_{k-1}^2(\lambda r_0) + Y_{k-1}^2(\lambda r_0)}}.
$$

Finally with the help of integral transform we reduced the Helmholtz equation to the following integral relation
\begin{align*}
\hat w_k(t, \lambda) = 
e^{-\lambda^2 t} \hat w_{0,k}(\lambda) + r_k(\lambda) \int_0^t e^{-\lambda^2(t-\tau)}u_k(\tau)d\tau \nonumber \\ +  \int_0^t e^{-\lambda^2(t-\tau)}W_{k,k-1}[(\tilde \bv,\nabla)w]_k(\tau, \lambda)d\tau,
\end{align*}
where $\hat w_k(t, \cdot) = W_{k,k-1}[w_k(t, \cdot)]$, $[(\tilde \bv,\nabla)w]_k$ - $k$-th Fourier coefficient of the term $(\tilde \bv,\nabla)w$. Take the inverse transform $W_{k,k-1}$ and rewrite it in terms of Stokes semigroup $S(t)$ corresponding to the problem  (\ref{stokeseq}),  (\ref{robin_bound}),  (\ref{boundinfvorticity}):
\begin{align}\label{_3:int_rel}
w(t,\bx) = 
S(t)w_0(\bx) + \sum_{k=-N}^N e^{ik\varphi} W_{k,k-1} \left [ r_k(\lambda) \int_0^t e^{-\lambda^2(t-\tau)}u_k(\tau)d\tau \right ] \nonumber \\ +  \int_0^t \left (\nabla S(t - \tau), \tilde \bv \right) w(\tau, \bx) d\tau.
\end{align}


\begin{thm}[resolvability of Oseen equation]
For given $\tilde \bv(t,\bx) \in C(R_+\times B_{r_0})\cap L_\infty^{loc}\left(R_+; L_\infty(B_{r_0})\right)$ satisfying no-slip condition (\ref{boundns}), $u_k(t) \in L_\infty^{loc}(\RM_+)$, $k=0,\dots, N$,
$w_0(\bx) \in L_2(B_{r_0})$,  the problem 
(\ref{oseeneqw}), (\ref{_3:initw})-(\ref{_3:bound3}) has a unique global solution $w(t,\bx) \in C \left ( [0,\infty), L_2(B_{r_0}) \right )$ which is locally Lipschitz mapping due to $\{u_k\}_{k=1}^N$, $w_0(\bx)$.
\end{thm}

\begin{proof}First we need prove the local resolvability of the equation (\ref{_3:int_rel}). Fix $T>0$. For given $\tilde \bv \in L_\infty(B_{r_0})$, $u_k\in C[0,T]$, $k=0,\dots, N$. Consider the map 
\begin{align*}
F\left(w(\tau, \cdot)\right) = S(t)w_0(\bx) + \sum_{k=-N}^N e^{ik\varphi} W_{k,k-1} \left [ r_k(\lambda) \int_0^t e^{-\lambda^2(t-\tau)}u_k(\tau)d\tau \right ] \nonumber \\ -  \int_0^t \left (\nabla S(t - \tau), \tilde \bv \right) w(\tau, \bx) d\tau.
\end{align*}

Consider the space $Q=C \left ( [0,T], L_2(B_{r_0}) \right )$. Since asymptotical behaviour of $r_k(\lambda)$ is $1/\sqrt \lambda$ then with some $C>0$
$$
\| r_k(\lambda) \int_0^t e^{-\lambda^2(t-\tau)}u_k(\tau)d\tau \|_{L_2(0,\infty, \lambda d\lambda)} \leq
C\|u_k(\tau)\|_{C[0,T]}.
$$ 

Set
$$M=\| w_0(\cdot)\|_{L_2} + C\sum_{k=-N}^N \|u_k(\tau)\|_{C[0,T]}.$$
Using estimates on $S(t)$ from Proposition \ref{stokes_semigroup_thm} we have
\begin{align*}
&\left \|F\left(w(t, \cdot)\right) \right\|_{L_2} \leq \| w_0(\cdot)\|_{L_2} + C\sum_{k=-N}^N \|u_k(\tau)\|_{C[0,T]} \nonumber + \\ &  \int_0^t \left \| (\nabla S(t - \tau), \tilde \bv )  w(\tau, \bx) \right\|_{L_2} d\tau \leq \\ & M +\|\tilde \bv\|_{L_\infty(R_+\times B_{r_0})} \|w\|_{Q} \int_0^t  \frac{d\tau}{\sqrt{e(t-\tau)}} = \\ &
M + 2 \sqrt \frac T e \|\tilde \bv\|_{L_\infty([0,T]\times B_{r_0})} \|w\|_{Q}.
\end{align*}

So we deduced that $F:Q \to Q$ is well-defined and maps $Q$ into itself. Now we prove that $F$ is a strict contraction in $Q$.
\begin{align*}
F\left(w_1(\tau, \cdot)\right) -  F\left(w_2(\tau, \cdot)\right) =  \int_0^t \left (\nabla S(t - \tau), \tilde \bv \right) \left (w_1(\tau, \bx) - w_2(\tau, \bx) \right )d\tau
\end{align*}
and
\begin{align*}
&\|F\left(w_1(\tau, \cdot)\right) -  F\left(w_2(\tau, \cdot)\right)\|_Q \leq \\ & \vraisup_{t\in[0,T]} \int_0^t \left\| (\nabla S(t - \tau), \tilde \bv ) \left (w_1(\tau, \bx) - w_2(\tau, \bx) \right )\right \|_{L_2}d\tau \leq \\ &
2\sqrt \frac T e \|\tilde \bv\|_{L_\infty([0,T]\times B_{r_0})} \|w_1-w_2\|_Q.
\end{align*}

Then for small $T$ by the Banach fixed point theorem, the map $F$ has a unique fixed point $w(t,\bx)$.

Next, we prove, that $\|w(t,\cdot)\|_{L_2(B_{r_0})}$ cannot blow up in finite time and the solution is global. Since $F(w)=w$ then from estimates above
\begin{align*}
\|w(t,\cdot)\|_{L_2} \leq M + \|\tilde \bv\|_{L_\infty([0,t]\times B_{r_0})} \int_0^t  \left ( \frac{1}{\sqrt{2e(t-\tau)}} \right )\|w(\tau,\cdot)\|_{L_2}d\tau
\end{align*}
and from Gronwall’s Lemma
\begin{align*}
\|w(t,\cdot)\|_{L_2} \leq M  e^{2 \sqrt \frac t e \vraisup_{\tau\in [0,t] }\|\tilde \bv(\tau, \cdot)\|_{L_\infty(B_{r_0})}}
\end{align*}
and $\|w(t,\cdot)\|_{L_2}$ stays finite for all time $t>0$.

No, we prove uniqueness. If $w_1(t,\cdot)$, $w_2(t,\cdot)$ are two solutions of (\ref{oseeneqw}), (\ref{_3:initw})-(\ref{_3:bound3}), then its difference $w_1-w_2$ is the solution of the same problem with zero initial and boundary conditions. Then $M$ defined above equals zero and from the same Gronwall’s Lemma
$$
\|w_1(t,\cdot)-w_2(t,\cdot)\|_{L_2} \leq 0,
$$
and the theorem is completely proved.
\end{proof}
\subsection{Existence theorem for Helmholtz equation in exterior of the disc} \label{ns_sect}
We consider the map 
\begin{align} \label{mapF}
F\left(w(t, \cdot)\right) = S(t)w_0(\bx) + \sum_{k=-N}^N e^{ik\varphi} W_{k,k-1} \left [ r_k(\lambda) \int_0^t e^{-\lambda^2(t-\tau)}u_k(\tau)d\tau \right ] \nonumber \\ +  \int_0^t  S(t - \tau)   (\bv,  \nabla w) d\tau
\end{align}
associated with the boundary-value problem for the Helmholtz equation (\ref{helmholtzeq}). In virtue of Lemma \ref{lembsest} $F$ will be well-defined in $Q=C \left ( [0,T], W) \right )$, where $W =  \{w \in H^1(B_{r_0})$, $w_{-1}, w_0, w_1 \in L_1(r_0,\infty;r) \}$. Denote cylinder $Z_{r_0,T} = [0,T]\times B_{r_0}$.

\begin{thm}[resolvability of Helmholtz equation]
For any $L>0$ there exists $T=T(L)>0$ such that for all $w_0 \in H^1(B_{r_0})$, $u_k\in C[0,T]$, $k=0,\dots, N$, $\|w_0\|_{H^1} \leq L$, $\| u_k(\cdot) \|_{C[0,T]}<L$ the problem (\ref{helmholtzeq}), (\ref{_3:initw})-(\ref{_3:bound3}) has a unique solution $w(t,\bx) \in Q$.
\end{thm}

\begin{proof}
First we derive that $F\left(w(t, \cdot)\right) \in H^1(B_{r_0})$ for fixed $t\in [0,T]$. From the proposition \ref{stokes_semigroup_thm} $S(t)$ is continuous mapping in $H^1(B_{r_0})$. 

From the estimate
$$
\| \lambda r_k(\lambda) \int_0^t e^{-\lambda^2(t-\tau)}u_k(\tau)d\tau \|_{L_2(0,\infty, \lambda d\lambda)} \leq
C\|u_k(\tau)\|_{C[0,T]}
$$ 
the second term in (\ref{mapF}) belongs to $H^1(B_{r_0})$.

With help of the proposition \ref{stokes_semigroup_thm} we have
\begin{align*}
\left \|F\left(w(t, \cdot)\right) \right\|_{H^1(B_{r_0})} \leq \sqrt 3 \| w_0(\cdot)\|_{H^1(B_{r_0})} + C\sum_{k=-N}^N \|u_k(\tau)\|_{ C[0,T]} \nonumber + \\  \int_0^t \left \|  S(t - \tau)  \left ( \bv, \nabla w(\tau, \bx) \right ) \right\|_{H^1(B_{r_0})} d\tau \leq \\
M + \int_0^t \left (1+\frac1{\sqrt {e(t-\tau)}} \right )\left \|\left ( \bv, \nabla w(\tau, \bx) \right ) \right\|_{H^1(B_{r_0})} d\tau
\end{align*}
with
$$M=\sqrt 3 \| w_0(\cdot)\|_{H^1} + C\sum_{k=-N}^N \|u_k(\tau)\|_{ C[0,T]}.$$

We invoke the following Sobolev inequality:
$$
\|\bv(t,\cdot) - \bv_\infty\|_{L_\infty(B_{r_0})} \leq C \|\bv - \bv_\infty(t,\cdot)\|^\frac 12 _{L_4(B_{r_0})} \|\nabla \bv(t,\cdot)\|^\frac 12_{L_4(B_{r_0})}.
$$

From (\ref{bsest2}) follows
$$
\|\nabla \bv(t,\cdot)\|_{L_4(B_{r_0})} \leq C \|w(t,\cdot) \|_{L_4(B_{r_0})}.
$$

Then in view of
$$
\|w(t,\cdot) \|_{L_4(B_{r_0})} \leq C \|w(t,\cdot)\|^\frac 12_{L_2(B_{r_0})} \|\nabla w(t,\cdot)\|^\frac 12_{L_2(B_{r_0})}
$$
we have
$$
\|\nabla \bv(t,\cdot)\|_{L_4(B_{r_0})} \leq C \|w(t,\cdot)\|^\frac 12_{L_2(B_{r_0})} \|\nabla w(t,\cdot)\|^\frac 12_{L_2(B_{r_0})}.
$$

For any $r$ on circle $S_r = \{\bx \in \RM^2,~|\bx|=r\}$, $r \geq r_0$ holds:
$$
\vraisup_{r\in [r_0,\infty)} \|\bv(t,\cdot) - \bv_\infty\|_{L_4(S_r)} \leq C \|\bv(t,\cdot) - \bv_\infty\|^\frac12_{L_2(S_r)} \|\nabla \bv(t,\cdot)\|^\frac12_{L_2(S_r)}
$$

Then from Lemma \ref{bsest}
\begin{align*}
&\int_{B_{r_0}} |\bv(t,\cdot) - \bv_\infty|^4d\bx \leq \\ &C \left (\|w(t,\cdot)\|^2_{L_2(B_{r_0})} + \sum_{k=-1,0,1}\|w_k(t,\cdot)\|^2_{L_1(r_0,\infty)} \right) \|\nabla \bv(t,\cdot)\|^2_{L_2(B_{r_0})}.
\end{align*}

Finally in virtue of (\ref{bsest2}) 
\begin{align*}
&\|\bv(t,\cdot) - \bv_\infty\|_{L_\infty(B_{r_0})} \leq \\
&C \left (\|w(t,\cdot)\|^2_{L_2(B_{r_0})} + \sum_{k=-1,0,1}\|w_k(t,\cdot)\|^2_{L_1(r_0,\infty)} \right)^\frac 18 \times \\ &\|w(t,\cdot)\|^\frac 12  \|\nabla w(t,\cdot)\|^\frac 14 _{L_2(B_{r_0})}
\end{align*}
and
\begin{align}\label{vwest}
\|\bv\|_{L_\infty(Z_{r_0,T})} \leq C \|w\|_Q.
\end{align}

So, with new constant $C>0$
\begin{align*}
&\left \| \int_0^t S(t - \tau) ( \bv , \nabla w) d\tau \right \|_{H_1(B_{r_0})} \leq  \\
&C \left(T+2\sqrt \frac Te \right) \|\bv(t,\cdot)\|_{L_\infty(Z_{r_0,T})} \|w(\tau, \bx)\|_{C \left ( [0,T], H^1(B_{r_0}) \right )}
\end{align*}
and $F\left(w(t, \cdot)\right) \in H^1(B_{r_0})$.

Now we prove that the first Fourier coefficients for $k=-1, 0, 1$ of the function $F\left(w(\tau, \cdot)\right)$ belong to $L_1(r_0,\infty;r)$. 

From (\ref{robin_bound}) zero Fourier coefficient of the Stokes semi-group $[S(t)]_0$ generates radially symmetrical solution $w(t,\bx)$ of the heat equation with Newman boundary condition
$$
\frac{\partial w(t,\bx')}{\partial n} = 0,~|\bx'|=r_0.
$$
  
Semigroup coefficients $[S(t)]_{\pm 1}$ correspond to solutions of the heat equation with Robin boundary
$$
\frac{\partial w}{\partial n} + w(t,\bx')  = 0,~|\bx'|=r_0.
$$

From $L_p-L_q$ estimates for heat equation \cite{Dv} with some $C>0$
$$
\left \| [S(t)]_k f \right \|_{L_1(r_0,\infty, r)} \leq C\sqrt t \|f\|_{L_2(r_0,\infty, r)} 
$$
for $k=-1,0,1$.

So we deduced that $F:Q \to Q$ is well-defined and maps $Q$ into itself. Now we prove that $F$ is a strict contraction in $Q$ in some ball $B=\{ w \in Q~|~||w||_Q<L\}$. 

Take $w_1, w_2 \in B$ with corresponding velocity fields $\bv_1$, $\bv_2$:
\begin{align*}
&F\left(w_1(\tau, \cdot)\right) -  F\left(w_2(\tau, \cdot)\right) =\\  &\int_0^t  S(t - \tau)  \left ( \bv_1 , \nabla \left (w_1(\tau, \bx) - w_2(\tau, \bx) \right ) \right )d\tau +\\
&\int_0^t  S(t - \tau) \left ( \bv_1 - \bv_2 , \nabla w_2(\tau, \bx) \right ) d\tau 
\end{align*}

Then from (\ref{vwest}) with some constants $C_1$, $C_2$
\begin{align*}
&\|F\left(w_1(\tau, \cdot)\right) -  F\left(w_2(\tau, \cdot)\right)\|_{H^1} \leq \\ &\vraisup_{t\in[0,T]} \int_0^t \left\|  S(t - \tau)  \left ( \bv_1 , \nabla \left (w_1(\tau, \bx) - w_2(\tau, \bx) \right ) \right ) \right \|_{H^1}d\tau + \\& \vraisup_{t\in[0,T]} \int_0^t \left\|   S(t - \tau) \left ( \bv_1 - \bv_2 , \nabla w_2(\tau, \bx) \right ) \right \|_{H^1}d\tau \leq \\&
\left (T+2\sqrt \frac T e \right )\left (\| 
 \bv_1\|_{L_\infty(Z_{r_0,T})} \|w_1-w_2\|_Q +  \|\bv_1 - \bv_2\|_{L_\infty(Z_{r_0,T})} \|w_2\|_Q\right )\leq \\&
 C_1\left (T+2\sqrt \frac T e \right )\left (\| 
 \bv_1\|_{L_\infty(Z_{r_0,T})} +  \|w_2\|_Q\right )  \|w_1-w_2\|_Q
 \leq \\& C_2 L \left (T+2\sqrt \frac T e \right )  \|w_1-w_2\|_Q.
\end{align*}

Estimates of Fourier coefficients for $k=-1,0,1$ can be held in a similar way using $L_p-L_q$ estimates. So, for any $L>0$ we can find such a small $T>0$ that $F$ becomes the strict contraction map in $Q$. Then for small $T$ by the Banach fixed point theorem, the map $F$ has a unique fixed point $w(t,\bx)$. The theorem is completely proved.
\end{proof}

\end{document}